\documentclass{amsart}
\usepackage{amsmath, amssymb, amsthm, amscd, amsfonts, eucal, hyperref}
\usepackage{enumerate}
\usepackage[T5, T1]{fontenc}
\usepackage[all]{xy}
\usepackage{charter} 
 \usepackage{xcolor}
 
\makeatletter
\@namedef{subjclassname@2020}{%
  \textup{2020} Mathematics Subject Classification}
\makeatother

\numberwithin{equation}{section}
\numberwithin{figure}{section}

\newtheorem*{theoremT}{Theorem \ref{firstTachikawa}}
\newtheorem*{theoremG}{Theorem \ref{characterizationGolod}}
\newtheorem{theorem}{Theorem}[section]
\newtheorem*{propositionM}{Proposition \ref{monotonicity}}
\newtheorem{proposition}[theorem]{Proposition}
\newtheorem{lemma}[theorem]{Lemma}
\newtheorem {corollary}[theorem]{Corollary}

\theoremstyle {definition}
\newtheorem {definition}[theorem]{Definition}
\newtheorem {example}[theorem]{Example}
\newtheorem {question}[theorem]{Question}
\newtheorem {problem}[theorem]{Problem}
\newtheorem {conjecture}[theorem]{Conjecture}
\theoremstyle {remark}
\newtheorem{remark}[theorem]{Remark}

\DeclareMathOperator{\Hom}{Hom}

\DeclareMathOperator{\Ext}{Ext}
\DeclareMathOperator{\Tor}{Tor}
\DeclareMathOperator{\Ima}{Im}
\DeclareMathOperator{\Ker}{Ker}
\DeclareMathOperator{\Cok}{Coker}
\DeclareMathOperator{\depth}{depth}

\DeclareMathOperator{\soc}{soc}
\DeclareMathOperator{\syz}{syz}

\newcommand{\fm}{\ensuremath{\mathfrak m}}
\newcommand{\fn}{\ensuremath{\mathfrak n}}

\newcommand{\bN}{\ensuremath{\mathbb N}}

\newcommand{\rh}{\ensuremath{\mathrm h}}
\newcommand{\rH}{\ensuremath{\mathrm H}}

\begin{document}

\title[On direct summands of syzygies of the residue field of a local ring]{On direct summands of syzygies of the residue field of a local ring}

\author{\fontencoding{T5}\selectfont \DJ o\`an Trung C\uhorn{}\`\ohorn ng} \address{\fontencoding{T5}\selectfont \DJ o\`an Trung C\uhorn{}\`\ohorn ng, Institute of Mathematics, Vietnam Academy of Science and Technology, 18 Hoang Quoc Viet, 10072 Hanoi, Viet Nam.} \email{dtcuong@math.ac.vn}

\author{Toshinori Kobayashi }
\address{Toshinori Kobayashi, Department of Mathematics, School of Science and Technology, Meiji University, 1-1-1 Higashi- mita, Tama-ku, Kawasaki 214-8571, Japan.}
\email{tkobayashi@meiji.ac.jp}


\thanks{\fontencoding{T5}\selectfont \DJ o\`an Trung C\uhorn{}\`\ohorn ng was funded by Vietnam Academy of Science and Technology under the project code CTTH00.01/25-26.
Toshinori Kobayashi was partly supported by JSPS Grant-in-Aid for Early-Career Scientists 25K17240.}



\begin{abstract}
We investigate local rings in which a syzygy of the residue field occurs as a direct summand of another syzygy of the field. This class of local rings includes Golod rings, Burch rings and non-trivial fiber products of local rings. For such rings, we prove that the Betti sequence of any finitely generated module is eventually periodically non-decreasing. As an application, we confirm the Tachikawa conjecture for all Cohen-Macaulay local rings satisfying this syzygy condition. In the second part of the paper, we show that a recursive direct sum decomposition of the syzygy of the residue field characterizes Golod rings, thereby establishing the converse to a recent theorem of Cuong-Dao-Eisenbud-Kobayashi-Polini-Ulrich \cite{CDEKPU25}.
\end{abstract}
\maketitle


\section{Introduction}

Let $(R, \fm, k)$ be a Noetherian local ring. We denote the $n$-th syzygy of an $R$-module $M$ by $\syz^R_n(M)$. In commutative algebra, the study of direct sum decompositions and direct summands of syzygies of the residue field of local rings is an interesting topic. Any information about such decompositions and corresponding summands would lead to remarkable insights into the (co)homology and singularity of the ring. 

For examples, Dutta \cite[Corollary 1.3]{Dut89} showed that $R$ is regular if and only if a syzygy $\syz^R_n(k)$ contains a nonzero free direct summand for some $n\geq 0$. This result motivates a broad program of identifying special direct summands of syzygies and applying them to characterize ring properties. Extending Dutta's theorem, Takahashi \cite[Theorem 4.3]{Tak06}  proved that $R$ is regular if and only if a syzygy $\syz^R_n(k)$ contains a semi-dualizing direct summand for some $n$. In the same spirit, he further showed that $R$ is Gorenstein if and only if a syzygy $\syz^R_n(k)$ contains a non-zero G-projective direct summand for some $n\leq \depth(R)+2$ (see \cite[Theorem 6.5]{Tak06}). For Cohen-Macaulay rings, Snapp \cite[Theorem 2.4]{Sna10} proved that $R$ is Cohen-Macaulay if and only if there exists a system of parameters $\underline x$ of $R$, which is part of a minimal generating set of $\fm$, such that a syzygy $\syz^R_n(R/\underline x)$ contains a non-zero free direct summand. 

On the other hand, if some syzygies of the residue field contain a simple direct summand, that is, a summand isomorphic to $k$, then one may expect $R$ to have a complicated structure. Such rings are called non-exceptional in the sense of Lescot \cite{Les85}. Nonetheless, over a non-exceptional local ring, Lescot \cite[Proof of Theorem B]{Les85} proved that some syzygies of a finitely generated $R$-module always contain a simple direct summand. For these rings, one can confirm Avramov's question on the eventual growth of the Betti sequence of general modules. Recently, further subclasses of non-exceptional rings have been identified. Dao-Kobayashi-Takahashi \cite[Theorem 4.1]{DKT20} showed that $R$ is a Burch ring of depth zero if and only if $\syz^R_2(k)$ contains a simple direct summand. Furthermore, DeBellevue-Miller \cite[Theorem A]{DM}, extending the result of Dao-Eisenbud \cite[Theorem 0.1]{DE22},  proved that if $R$ has depth zero and Burch index at least $2$, then the syzygies $\syz^R_n(M)$ of any finitely generated $R$-module $M$ contain a simple direct summand for all $n\geq 5$.

In general, it is quite difficult to identify non-trivial direct sum decomposition of syzygies. If $R$ is a Gorenstein Artinian local ring, then it is well-known that the syzygies of any indecomposable module are indecomposable (see \cite[Lemma 8.17]{Yos90}). In particular, the modules $\syz^R_n(k)$'s are indecomposable. In sharp contrast to the Gorenstein case, Cuong-Dao-Eisenbud-Kobayashi-Polini-Ulrich recently discovered the following decomposition for Golod rings \cite[Theorem 1.1]{CDEKPU25}.

\begin{theorem}\label{generalGolod}
Let $(R,\fm,k)$ be a Noetherian local ring of embedding dimension $e$ and let $K_{\bullet}$ be the Koszul complex on a minimal set of generators of $\fm$. If $R$ is Golod then 
\begin{equation} \label{generalGolod22}
\syz_{e+1}^R(k) \simeq \bigoplus_{i = 0}^{e-1} \syz_i^R(k)^{h_{e-i}}
\end{equation}
and, more generally,
$$
\syz_{e+1+j}^R(k) \simeq \bigoplus_{i=0}^{e-1} \syz_{i+j}^R(k)^{h_{e-i}},
$$
for any $j\geq 0$, where $h_i = \dim_k(H_i(K_{\bullet}))$. 
\end{theorem}

This simple recursive formula shows that higher syzygies of Golod rings are highly decomposable. Thus, the decomposability of $\syz^R_n(k)$ can indicate non-Gorenstein behavior, as in the case of Golod rings.

It is worth noting that special direct summand of modules other than $k$ have also been studied by Lescot \cite{Les85}, Christensen-Striuli-Veliche \cite{CSV10}, Snapp \cite{Sna10}, Ghosh \cite{Gho19}, Dao-Eisenbud \cite{DE22}, among others.

In this paper, we address two problems concerning direct summands and direct sum decompositions of syzygies of the residue field of Noetherian local rings.

In the first problem, we consider local rings which satisfy the following condition:

\medskip
$(\ast)$ A syzygy of the residue field is a direct summand of another syzygy of that field.
\medskip

For convenience, we refer to $(\ast)$ as the {\it syzygy direct summand assumption}.

Examples of local rings satisfying the syzygy direct summand assumption include those for which a syzygy of the residue field contains a simple direct summand. Golod rings satisfy $(\ast)$ as a consequence of Theorem \ref{generalGolod}. In \cite[Proposition 5.10, Theorem 4.1]{DKT20}, Dao, Kobayashi and Takahashi show that if a local ring $R$ is Burch of depth $t$, then $\syz^R_t(k)$ is a direct summand of $\syz^R_{t+2}(k)$. Thus, Burch rings also satisfy the syzygy direct summand assumption $(\ast)$. We will provide further examples of such rings in the next section.

Our aim is to investigate Betti growth, $\Ext$-vanishing, and their relation to the singularity of $R$ under the assumption $(\ast)$. As a first step, we establish a result on the monotonicity of the Betti sequence.

\begin{propositionM}
Let $(R, \fm, k)$ be a Noetherian local ring. Suppose there are non-negative integers $a\not=b$ such that the $a$-th syzygy $\syz_a^R(k)$ is a direct summand of $\syz_b^R(k)$. Then either $a<b$ or $R$ is a hypersurface. 

Suppose that $R$ is not a hypersurface.
\begin{itemize}
\item[(a)] Denote $p=b-a$ and take an integer $0<q\leq b-a$. For a finitely generated $R$-module $M$, the subsequence of Betti numbers $\{\beta_{pn+q}(M)\}_{pn+q> a}$ is non-decreasing.
\item[(b)] $\depth R \le a$.
\end{itemize}
\end{propositionM}

Thus, the Betti sequence of any finitely generated $R$-module is periodically non-decreasing. Avramov \cite[Problem 4.3.3]{Avr98} asked whether the Betti sequence of any finitely generated module over local rings is eventually non-decreasing. In Theorem \ref{firstTachikawa}, we give an affirmative answer to a weaker version of this question in a particular case.

Proposition \ref{monotonicity} allows us to confirm Tachikawa conjecture for all Cohen-Macaulay local rings that satisfy the syzygy direct summand assumption $(\ast)$.

\begin{theoremT}
Let $(R, \fm, k)$ be a Cohen-Macaulay local ring with a canonical module $K_R$. Suppose there are non-negative integers $a\not= b$ such that the $a$-th syzygy $\syz_a^R(k)$ is a direct summand of $\syz_b^R(k)$. If $\Ext^i_R(K_R, R)=0$ for all $i>0$ then $R$ is a hypersurface.
\end{theoremT}

The second problem concerns the converse of Theorem \ref{generalGolod}, namely, whether the decomposition (\ref{generalGolod22}) of the syzygy $\syz^R_{e+1}(k)$ characterizes the Golod property of the ring $R$. We prove that this is indeed the case.

\begin{theoremG}
Let $(R,\fm,k)$ be a Noetherian local ring of embedding dimension $e$ and codepth $c:=e-\depth(R)$, and let $K_{\bullet}$ be the Koszul complex on a minimal set of generators of $\fm$.
Put $\rh_i(R)=\dim_k \mathrm{H}_i(K_\bullet)$.
The following statements are equivalent.
\begin{itemize}
\item[(a)] The ring $R$ is Golod.
\item[(b)] There is an isomorphism
$$\syz_{e+1}^R(k) \simeq \bigoplus_{j = 1}^{c} \syz_{e-j}^R(k)^{\rh_{j}(R)}.$$
\item[(c)] For each $m\ge 0$, there are isomorphisms
$$\syz_{e+1+m}^R(k) \simeq \bigoplus_{j=1}^{c} \syz_{e-j+m}^R(k)^{\rh_{j}(R)}.$$
\end{itemize}
\end{theoremG}

The paper is organized into three sections. In Section 2, we provide several examples of local rings that satisfy the syzygy direct summand assumption $(\ast)$. In particular, Burch rings, Golod rings, and non-trivial fibre products of local rings satisfy this assumption. We prove Proposition \ref{monotonicity} and Theorem \ref{firstTachikawa} in this section. Section 3 is devoted to the proof of Theorem \ref{characterizationGolod}.

\bigskip
\noindent{\bf Acknowledgment.}  The authors 
thank Ryo Takahashi for helpful discussions during the preparation of this manuscript. The second author wishes to express his gratitude to the Institute of Mathematics, Vietnam Academy of Science and Technology (IMVAST), for its hospitality and support through the Simons Foundation Targeted Grant for the IMVAST (No. 558672), and to the International Center for Research and Training in Mathematics under the auspices of UNESCO (ICRTM-IMVAST) during his visit to the institute.


\section{Special syzygy direct summands, Betti growth, (co)homology vanishing and singularity}

Throughout this paper, $(R, \fm, k)$ denotes a Noetherian local ring. We denote by $e$ the embbedding dimension $\dim_k(\fm/\fm^2)$. 

In this section, we present several classes of Noetherian local rings that satisfy the syzygy direct summand assumption $(\ast)$. We then discuss the problem of the growth of Betti numbers of a module over such rings and its application to Tachikawa conjecture.

The first important class consists of Noetherian local rings in which some syzygy of the residue field contains a simple direct summand. Such rings are called non-exceptional. If $\syz_n^R(k)$ has a simple direct summand for some $n>0$, then $\syz_0^R(k)\simeq k$ is a direct summand of $\syz_n^R(k)$. Rings of this type include Burch rings of depth zero, as shown by Dao-Kobayashi-Takahashi \cite[Theorem 4.1]{DKT20}. Since Teter rings are Burch, they also belong to this class. For more on Teter rings, we refer to the works of Teter \cite{Tet74} and Huneke-Vraciu \cite{HV06}.

The second class consists of Burch rings, defined by Dao-Kobayashi-Takahashi \cite{DKT20}. It is shown that if $R$ is a Burch ring of depth $t$ then $\syz^R_t(k)$ is a direct summand of $\syz^R_{t+2}(k)$ (see \cite[Proposition 5.10, Theorem 4.1]{DKT20}). On the other hand, not every ring satisfying the assumption $(\ast)$ is a Burch ring. For instance, consider the ring 
$$R = k[[x, z]] / (x^4, x^2z^2, z^4),$$
where $k$ is a field. According to \cite[Examples 4.6 and 5.7]{DKT20}, the residue field $k$ is a direct summand of $\syz_3^R(k)$ but not of $\syz_2^R(k)$. This implies that $R$ is not a Burch ring, even though it satisfies the assumption $(\ast)$.

The third class consists of Golod rings. Let $\underline{x}$ be a minimal set of generators of the maximal ideal $\fm$. For a finitely generated $R$-module $M$, let $\mathrm{H}_i(\underline{x}; M)$ denote the $i$-th Koszul homology of $M$ with respect to $\underline{x}$, and define 
$$\rh_i^R(M)=\dim_k\mathrm{H}_i(\underline{x}; M).$$
Note that $\rh_i^R(M)$ is independent of the choice of $\underline{x}$. When the context is clear, we write simply $\rh_i(M)$. The Poincar\'e series of $k$ over $R$ is $P^R_k(t):=\sum_{i=0}^\infty\beta_R^i(k)t^i$. The following term-wise inequality for the Poincar\'e series is due to Serre.

\begin{equation}\label{equ21}
P^R_k(t)\le \frac{(1+t)^{e}}{1-t\sum_{j=1}^{e-\depth R}\rh_j(R)t^j}.
\end{equation}

\begin{definition}
We say that $R$ is \emph{Golod} if in \eqref{equ21} the equality holds.
\end{definition}

Golod rings satisfy the assumption $(\ast)$ as a consequence of Theorem \ref{generalGolod}.

Following Lescot \cite[Definition 3.1]{Les85}, a Noetherian local ring is called an \emph{exceptional local ring} if no syzygy of the residue field has a simple direct summand. An example of exceptional ring is given in \cite[Examples 3.8(2)]{Les85}. In contrast, Burch rings of depth zero and Golod rings of depth zero are not exceptional. Using fibre product of local rings, one can construct many examples of exceptional local rings that nevertheless satisfy the syzygy direct summand assumption $(\ast)$.

Let $(S,\fm_S,k)$ and $(T,\fm_T,k)$ be Noetherian local rings with the same residue field $k$. Denote by $\pi_S\colon S \to k$ and $\pi_T\colon T \to k$ the canonical surjections. The \emph{fibre product} of $S$ and $T$ over $k$ is the local ring
\[
R:=S \times_k T:=\{(a,b)\in S\times T \mid \pi_S(a)=\pi_T(b)\}.
\]
Assume $S\not=k \not=T$.  Then the maximal ideal $\fm_R$ decomposes into a direct sum of $R$-modules $\fm_R=\fm_S\oplus \fm_T$. Moreover, for each integer $n\ge 2$, there is a decomposition
\[
\syz_n^Rk \cong \left(\bigoplus_{i=1}^n\syz_i^Sk^{\oplus c_i}\right)\oplus \left(\bigoplus_{i=1}^n\syz_i^Tk^{\oplus d_i}\right)\cong \fm_R\oplus \cdots
\]
for some positive integers $c_1,\dots,c_n$ and $d_1,\dots,d_n$; see Nasseh-Takahashi \cite[Lemma 3.2]{NT20}. Hence $R$ satisfies the syzygy direct summand assumption $(\ast)$. Furthermore, $R$ is exceptional if and only if both $S$ and $T$ are exceptional.

There exist local rings that satisfy the syzygy direct summand assumption and are neither Golod nor non-trivial fibre products. Likewise, there are examples of exceptional local rings that satisfy the syzygy direct summand assumption as well.

\begin{example}
Let $k$ be a field.
Let $S=k[[x,y]]/(x^2,y^2)$, $T=k[[z,w]]/(z^2,w^2)$, and $R=S\times_k T\cong k[[x,y,z,w]]/(x^2,y^2,xz,xw,yz,yw,z^2,w^2)$.
The rings $S$ and $T$ are Gorenstein and are not hypersurfaces. Thus $R$ is exceptional and satisfies the syzygy direct summand assumption $(\ast)$.
\end{example}

\begin{example}
Let $R=\mathbb{Q}[[x,y,z]]/(x^3,y^3,z^3,xy,xz^2)$.
The ring $R$ is neither Golod nor a nontrivial fibre product. 
Computations by Macaulay2 show that $k$ is not a direct summand of $\syz_2^Rk$, but of $\syz_3^R k$.
\end{example}

\begin{example}
Let $R=\mathbb{Q}[[x,y,z]]/(x^3,y^3,z^3,x^2y,yz^2)$.
The ring $R$ is neither Golod nor a nontrivial fibre product.
Computations by Macaulay2 show that $k$ is not a direct summand of $\syz_2^Rk$ or $\syz_3^Rk$, but of $\syz_4^R k$.
\end{example}

Local rings satisfying the syzygy direct summand assumption $(\ast)$ possess interesting homological properties. To explore these properties, we recall a lemma on the behavior of syzygies under hyperplane sections (see \cite[Corollary 5.3]{Tak06}).

\begin{lemma}\label{l27}
Let $x\in \fm \setminus \fm^2$ be a non-zero divisor on $R$. For each integer $i\ge 1$, there is an isomorphism of $R/xR$-modules
\[
\syz_i^R k/x\syz_i^R k \cong \syz_i^{R/xR} k \oplus \syz_{i-1}^{R/xR} k.
\]
\end{lemma}

We denote by $\binom{n}{m}$ the binomial coefficient, with the convention that $\binom{n}{m}=0$ if $m<0$ or $m>n$. The following corollary is a straight forward consequence of Lemma \ref{l27}.

\begin{corollary} \label{c26}
Let $\underline{x}=x_1,\dots,x_t$ be  an $R$-sequence contained in $\fm\setminus \fm^2$.
Then for each $i\ge t$, there is an isomorphism of $R/\underline{x}R$-modules
\begin{equation*} 
\syz_i^R k/\underline{x}\syz_i^R k \cong \bigoplus_{j=0}^t\syz_{i-j}^{R/\underline{x}R} k^{\oplus \binom{t}{j}}.
\end{equation*}
\end{corollary}

The next proposition shows that the syzygy direct summand assumption $(\ast)$ and the Gorenstein property are quite different.

\begin{proposition}  \label{p29}
A Gorenstein local ring satisfies the syzygy direct summand assumption $(\ast)$ if and only if it is a hypersurface.
\end{proposition}
\begin{proof}
Let $R$ be a Gorenstein local ring. Suppose $R$ satisfies $(\ast)$. It means that there are distinct integers $a$, $b$ and an isomorphism 
$$\syz_a^Rk \cong \syz_b^Rk\oplus N$$
for some $R$-module $N$. Taking further syzygies of both sides if necessary, we may assume $a,b\ge \dim R$. Let $x_1,\dots,x_d$ be a maximal regular sequence contained in $\fm\setminus \fm^2$. Applying Lemma \ref{l27} iteratively, we obtain that 
$$\bigoplus_{i=0}^d \syz_{b-i}^{R'}k^{\oplus \binom{d}{i}}$$
is a direct summand of 
$$\bigoplus_{i=0}^d \syz_{a-i}^{R'}k^{\oplus \binom{d}{i}},$$
where $R':=R/(x_1,\dots,x_d)$ is an Artinian ring. We may also assume that $R'$ is not a field. In this case, for any $n\ge 0$, $\syz_n^{R'}k$ is indecomposable since $R$ is Artinian Gorenstein. Thus, for each $i\in \{0,\dots,d\}$, there exists $j\in \{0,\dots,d\}$ such that $\syz_{b-i}^{R'}k$ is isomorphic to $\syz_{a-j}^{R'}k$. 

Since $a\not=b$, there is at least one pair $(i,j)$ with $b-i\not=a-j$ and $\syz_{b-i}^{R'}k\cong \syz_{a-j}^{R'}k$. It follows that $k$ has an eventually periodic resolution over $R'$, and consequently $R$ is a hypersurface.

The converse is immediate.
\end{proof}

We now investigate the monotonicity of the Betti sequence and its application to Tachikawa conjecture.

Let $(R, \fm, k)$ be a Noetherian local ring, and let $N$ be a finitely generated $R$-module with a minimal free resolution
$$\cdots \stackrel{d_3}{\longrightarrow} R^{\oplus \beta_2}\stackrel{d_2}{\longrightarrow}  R^{\oplus \beta_1}\stackrel{d_1}{\longrightarrow} R^{\oplus \beta_0}\stackrel{d_0}{\longrightarrow} N.$$
Applying the functor $\Hom_R(-, R)$ to this resolution yields the complex
$$0\longrightarrow \Hom_R(N, R)\stackrel{d^{-1}}{\longrightarrow}R^{\oplus\beta_0}\stackrel{d^0}{\longrightarrow}R^{\oplus\beta_1} \stackrel{d^1}{\longrightarrow}R^{\oplus\beta_2}\stackrel{d^2}{\longrightarrow}\cdots$$
whose $i$-th cohomology module is $\Ext^i_R(N, R)$.

\begin{lemma}\label{resolution}
Let $N$ be a finitely generated $R$-module. Suppose $\Ext^i_R(N, R)=0$ for all $i>0$. There is an exact sequence
$$0\longrightarrow \Hom_R(N, R)\stackrel{d^{-1}}{\longrightarrow}R^{\oplus\beta_0}\stackrel{d^0}{\longrightarrow}R^{\oplus\beta_1} \stackrel{d^1}{\longrightarrow}R^{\oplus\beta_2}\stackrel{d^2}{\longrightarrow}\cdots$$
where $\beta_i:=\beta_i(N)$ is the $i$-th Betti number of $N$. Furthermore, $\Ima(d^i)\subseteq \fm R^{\oplus\beta_{i+1}}$ for all $i\geq 0$.
\end{lemma}
\begin{proof}
The exactness follows from the discussion above. 

In the minimal free resolution of $N$, the map $d_{i+1}: R^{\oplus\beta_{i+1}}\rightarrow R^{\oplus \beta_i}$ is given by a matrix $A_{i+1}$ with entries in $\fm$. The corresponding matrix of the dual map
$$d^i: \Hom_R(R^{\oplus\beta_i}, R)\rightarrow \Hom_R(R^{\oplus \beta_{i+1}}, R)$$
is the transition matrix $A_{i+1}^T$. So $\Ima(d^i)\subseteq \fm R^{\oplus\beta_{i+1}}$.
\end{proof}

Combining the exact sequence in Lemma \ref{resolution} with a minimal free resolution of $\Hom_R(N, R)$, we obtain a two-side unbounded exact sequence of free $R$-modules
\begin{equation}\label{unboundedseq_h}
\cdots\stackrel{\delta_3}{\longrightarrow}R^{\oplus\alpha_2}\stackrel{\delta_2}{\longrightarrow}R^{\oplus\alpha_1} \stackrel{\delta_1}{\longrightarrow}R^{\oplus\alpha_0}\stackrel{\delta}{\longrightarrow}
R^{\oplus\beta_0}\stackrel{d^0}{\longrightarrow}R^{\oplus\beta_1} \stackrel{d^1}{\longrightarrow}R^{\oplus\beta_2}\stackrel{d^2}{\longrightarrow}\cdots
\end{equation}
where
\begin{itemize}
\item $\beta_i:=\beta_i(N)$ is the $i$-th Betti number of $N$;
\item $\alpha_i:=\beta_i(\Hom_R(N,R))$ is the $i$-th Betti number of $\Hom_R(N, R)$;
\item $\Ima(d^i)\subseteq \fm R^{\oplus\beta_{i+1}}$ for all $i\geq 0$;
\item $\Ima(\delta_i)\subseteq \fm R^{\oplus\alpha_{i+1}}$ for all $i\geq 0$;
\item $\Ima(\delta)\simeq \Hom_R(N, R)\simeq \Ker(d^0)$;
\item $\Ima(d^i)=\Hom_R(\syz_{i+1}^R(N), R)$.
\end{itemize}

\begin{lemma}\label{BettisequenceSyzygy}
Let $N$ be a finitely generated $R$-module. Suppose $\Ext^i_R(N, R)=0$ for all $i>0$. Put $\beta_0^\prime=\dim_k(\Cok(\delta\otimes_Rk)), \alpha_0^\prime=\dim_k(\Ker(\delta\otimes_Rk)$. The sequence 
\begin{equation}\label{Bettisequence}
\beta_n, \beta_{n-1}, \ldots,\beta_2, \beta_1, \beta_0^\prime, \alpha_0^\prime, \alpha_1, \alpha_2, \alpha_3, \ldots, \alpha_m, \ldots
\end{equation}
bounds above the Betti sequence of the module $\Hom_R(\syz_{n+1}^R(N), R)$. More precisely, we have
\begin{itemize}
\item $\beta_i=\beta_{n-i}^R(\Hom_R(\syz_{n+1}^R(N), R))$, for $0<i\leq n$;
\item $\beta_0^\prime\geq \beta_n^R(\Hom_R(\syz_{n+1}^R(N), R))$;
\item $\alpha_0^\prime\geq \beta_{n+1}^R(\Hom_R(\syz_{n+1}^R(N), R))$;
\item $\alpha_i\geq\beta_{n+i+1}^R(\Hom_R(\syz_{n+1}^R(N), R))$, for $i>0$.
\end{itemize}
In particular, $\beta_n, \beta_{n-1}, \ldots,\beta_2, \beta_1$ are the first $n$ Betti numbers of $\Hom_R(\syz_{n+1}^R(N), R)$ and 
$$\alpha_0-\alpha_0^\prime=\beta_0-\beta_0^\prime\geq 0.$$
 \end{lemma}
 \begin{proof}
 The conclusion is immediate from Lemma \ref{resolution} and the discussion above. 
 \end{proof}
  
If the Betti sequences of all finitely generated $R$-modules are non-decreasing, then the sequence (\ref{Bettisequence}) either stabilizes at a constant value or terminates after finitely many steps. Hence, the Betti sequence of $N$ is either constant, or $N$ has finite projective dimension.

In \cite{Avr84} (also \cite[Problem 4.3.3]{Avr98}), Avramov conjectured that the sequence of Betti numbers of any finitely generated module over a Noetherian local ring is eventually non-decreasing. The following proposition provides a partial confirmation of this conjecture for local rings satisfying the syzygy direct summand assumption $(\ast)$.

\begin{proposition}\label{monotonicity}
Let $(R, \fm, k)$ be a Noetherian local ring. Suppose there are non-negative integers $a\not=b$ such that the $a$-th syzygy $\syz_a^R(k)$ is a direct summand of $\syz_b^R(k)$. Then either $a<b$ or $R$ is a hypersurface. 

Suppose that $R$ is not a hypersurface.
\begin{itemize}
\item[(a)] Denote $p=b-a$ and take an integer $0<q\leq b-a$. For a finitely generated $R$-module $M$, the subsequence of Betti numbers $\{\beta_{pn+q}(M)\}_{pn+q> a}$ is non-decreasing.
\item[(b)] $\depth R \le a$.
\end{itemize}
\end{proposition}
\begin{proof}
First we assume that $a<b$. Denote $p=b-a$ and take an integer $0<q\leq b-a$.  By the assumption there is a decomposition $\syz_b^R(k)\simeq \syz_a^R(k)\oplus N$ where $N$ is a finitely generated $R$-module. Take $n\geq 0$ such that $p(n+1)-b+q=pn-a+q$ is positive. We have
\[\begin{aligned}
\beta_{p(n+1)+q}(M)&=\dim_k\Tor_{p(n+1)+q}^R(k, M)\\
&=\dim_k\Tor_{p(n+1)-b+q}^R(\syz_b^R(k), M) \\
&=\dim_k\Tor_{p(n+1)-b+q}^R(\syz_a^R(k), M)+\dim_k\Tor_{p(n+1)-b+q}^R(N, M)\\
&=\beta_{pn+q}(M)+\dim_k\Tor_{pn-a+q}^R(N, M)\\
&\geq \beta_{pn+q}(M).
\end{aligned}\]
It proves (a). 

If $b<a$ then a similar proof shows that the subsequence $\{\beta_{n(a-b)+q}(M)\}$ is eventually non-increasing, where $0<q\leq a-b$ is fixed. On the other hand, if $R$ is not a hypersurface then $k$ is extremal, i.e., $\lim_m\sup\beta_m(k)=\infty$, a contradiction!

Next we verify (b). Let $t=\depth R$. Note that $\depth \syz_b^R(k) \ge \min\{\depth R, b\}$.
Take an $R$-regular sequence $\underline{x}$ of length $t$. Then $\beta_t(R/(\underline{x}))=1< \beta_{t-p}(R/\underline{x})$.
\end{proof}

Suppose $R=S/J\fn$ where $(S, \fn, k)$ is a Noetherian local ring, $J$ is an ideal such that $J\fn\not=0$, then $R$ is Burch due to Dao-Kobayashi-Takahashi \cite[Example 2.2]{DKT20}. In particular, $\syz_t^R(k)$ is a direct summand of $\syz_{t+2}^R(k)$ where $t=\depth(R)$. Proposition \ref{monotonicity} confirms that for any finitely generated $R$-module $M$, the two sequences of Betti numbers $\{\beta_{2n}(M)\}_{n>0}$ and $\{\beta_{2n+1}(M)\}_{n\geq0}$ are non-decreasing. This result was actually proved long ago by Lescot \cite[Theorem A.3]{Les85}. So we can consider Proposition \ref{monotonicity} as a generalization of Lescot's theorem.

In ring theory, several homological conjectures investigate the relationship between the vanishing of certain $\Ext$ modules and the structural properties of the underlying ring. Notable examples include the Nakayama conjecture, the generalized Nakayama conjecture (due to  Auslander and Reiten), and two conjectures proposed by Tachikawa. In \cite{ABS05}, Avramov, Buchweitz, and Şega introduced a commutative version of Tachikawa's first conjecture on the vanishing of the module $\Ext^i_R(K_R, R)$.

\begin{conjecture}[Tachikawa]\label{Tachikawa}
Let $(R, \fm, k)$ be a Noetherian local ring with a canonical module $K_R$. If $\Ext^i_R(K_R, R)=0$ for all $i>0$ then $R$ is Gorenstein.
\end{conjecture}

The conjecture has been verified for several important classes of local rings, including (1) Artinian rings of Loewy length at most three by Asashiba \cite{Asa91}, Avramov-Buchweitz-\c{S}ega \cite{ABS05}, Huneke-\c{S}ega-Vraciu \cite[Theorem 2.11]{HSV04}); (2) Deformations of generically Gorenstein rings by Avramov-Buchweitz-\c{S}ega \cite[Theorem 3.1]{ABS05}, Hanes-Huneke \cite[Corollary 2.2]{HH05}; Golod rings by Avramov-Buchweitz-\c{S}ega \cite[Theorem 8.1]{ABS05}. If $R$ is a Burch ring then Dao-Kobayashi-Takahashi \cite[Proposition 7.12, Corollary 7.13]{DKT20} proved that $R$ is $\Tor$-friendly. Consequently, Tachikawa conjecture holds true for Burch rings. In all the works, it is shown that $\Ext$-vanishing is deeply related to the growth of Betti numbers and to the singularity of the ring. The connection becomes more precise when some syzygies of the residue field contain special direct summands. 

In what follows, we confirm Tachikawa conjecture \ref{Tachikawa} for local rings that satisfy the syzygy direct summand assumption $(\ast)$. It is the main result of the section.

\begin{theorem}\label{firstTachikawa}
Let $(R, \fm, k)$ be a Cohen-Macaulay local ring with a canonical module $K_R$.
Suppose there are non-negative integers $a\not= b$ such that the $a$-th syzygy $\syz_a^R(k)$ is a direct summand of $\syz_b^R(k)$. If $\Ext^i_R(K_R, R)=0$ for all $i>0$ then $R$ is a hypersurface.
\end{theorem}
\begin{proof}
Suppose $R$ is not a hypersurface. Then $a<b$. Let's write 
\[
\syz_b^R(k)\simeq \syz_a^R(k)\oplus N.
\]

We denote $p:=b-a$ and take an integer $0<q\leq p$. The subsequence $\{\beta_{pn+q}(K_R)\}_{pn+q>a}$ of the Betti sequence of the canonical module is non-decreasing due to Proposition \ref{monotonicity}. Let $m=pn+q$ for some $n$. By Lemma \ref{BettisequenceSyzygy}, the numbers $\beta_m(K_R)$, $\ldots$, $\beta_1(K_R)$ in this order are the first $m$ Betti numbers of $\Hom_R(\syz_{m+1}^R(K_R), R)$. As it holds true for any $m> 0$, the subsequence $\beta_m(K_R)$, $\beta_{m-p}(K_R)$, $\ldots$, $\beta_{p+q}(K_R)$, $\beta_{q}(K_R)$ is non-decreasing. It shows that $\{\beta_{pn+q}(K_R)\}_{pn+q>a}$ is constant. 

In the proof of Proposition \ref{monotonicity}, we have seen the equality
\[
\beta_{p(n+1)+q}(M)=\beta_{pn+q}(M)+\dim_k\Tor_{pn-a+q}^R(N, M).
\]
for $pn+q>a$ where $M$ is a finitely generated $R$-module and $\syz^R_b(k)\simeq \syz^R_a(k)\oplus N$. So 
\[
\Tor_m^R(N, K_R)=0
\]
for any $m>0$.
Let $F_\bullet$ be a minimal free resolution of $N$. Denote by $d$ the dimension of $M$. Applying the functor $(-)\otimes_R K_R$ to the exact sequence $0 \to \syz_d^R(N) \to F_{d-1} \to \cdots \to F_1 \to F_0$, we get an exact sequence
\[
0 \to \syz_d^R (N)\otimes_R K_R \to F_{d-1}\otimes_R K_R \to \cdots \to F_1\otimes_R K_R \to F_0\otimes_R K_R.
\]
In particular, it follows from depth lemma that $\syz_d^R(N)\otimes_R K_R$ is maximal Cohen-Macaulay.
Also, we further obtain an exact sequence
\[
\cdots \to F_{d+1}\otimes_R K_R \to F_d\otimes_R K_R \to \syz_d^R(N)\otimes_R K_R \to 0,
\]
where all terms are maximal Cohen-Macaulay.
Applying the functor $\Hom_R(-,K_R)$, we see that 
\[
0 \to \Hom_R(\syz_d^R(N),R) \to \Hom_R(F_d,R) \to \Hom_R(F_{d+1},R) \to \cdots 
\]
is exact. It means that $\Ext^{m}_R(\syz_d^R(N),R)=0$ for all $m > 0$. Using Lemma \ref{BettisequenceSyzygy} and Proposition \ref{monotonicity} again, a similar argument for $\syz_d^R(N)$ in place of $K_R$ shows that the Betti subsequence $\{\beta_{pn+q}(\syz_d^R(N))\}_{n\geq 0}$ is constant for given $0<q\leq p$. On the other hand, by a result of Avramov \cite[Theorem 8 and Corollary 9]{Avr96}, non-zero quotients of a syzygy of the residue field are extremal, in particular, their Betti sequences are unbounded. Hence $\syz_d^R(N)=0$. 

The syzygy of the residue field is therefore eventually periodic. This is possible if and only if $R$ is a hypersurface.
\end{proof}

In the proof of Theorem \ref{firstTachikawa} we have shown that the Betti sequence of the canonical module is $p$-periodically constant. In particular, the sequence is bounded. It is actually sufficient.

\begin{theorem}\label{secondTachikawa}
Let $(R, \fm, k)$ be a Cohen-Macaulay local ring with a canonical module $K_R$. Suppose there are non-negative integers $a\not= b$ such that the $a$-th syzygy $\syz_a^R(k)$ is a direct summand of $\syz_b^R(k)$. If the canonical module has bounded Betti numbers then $R$ is a hypersurface.
\end{theorem}
\begin{proof}
The proof is a refinement of the proof of Theorem \ref{firstTachikawa}, Lemma \ref{resolution} and Lemma \ref{BettisequenceSyzygy}. 

Suppose $R$ is not a hypersurface. Then $a<b$. The Betti sequence of the canonical module $K_R$ is bounded and non-decreasing by Proposition \ref{monotonicity}, it is then $(b-a)$-periodically constant from some point. Then $\Tor_m^R(N, K_R)=0,$ and hence $\Ext^m_R(N, R)=0$ for $m\gg0$. 

A similar proof as the proofs of Lemma \ref{resolution} and Lemma \ref{BettisequenceSyzygy} implies that the Betti subsequence $\{\beta_{pn+q}(N): n\in \bN, pn+q>a\}$ is eventually constant for given $0<q\leq p$. The extremity of $N$ as a quotient of some syzygy induces that $N=0$. Therefore $R$ is a hypersurface.
\end{proof}

For a local ring $R$ with a canonical module $K_R$, Christensen-Striuli-Veliche \cite[Question 1.3]{CSV10} asked whether the sequence $\{\beta_i(K_R)\}_{i>0}$ is eventually increasing, provided $R$ is not Gorenstein. Theorem \ref{firstTachikawa} gives a positive answer to this question for rings that satisfy the syzygy direct summand assumption $(\ast)$.

Theorem \ref{firstTachikawa} and Theorem \ref{secondTachikawa} strengthen the conclusion of Tachikawa conjecture in the case of local rings satisfying the syzygy direct summand assumption $(\ast)$. Such rings are also likely to satisfy the $\Tor$-persistence and $\Tor$-friendly properties in the sense of \cite{AINS22}. 

\begin{remark}
Let $(R, \fm, k)$ be an Artinian local ring with $\fm^3=0$.
\begin{itemize}
\item[(a)] If $R$ is not exceptional then any finitely generated $R$-module $M$ has a syzygy $\syz_i^R(M)$ containing a simple direct summand (see Lescot \cite[Proof of Theorem B]{Les85}). It is worth noting that the exceptional ring in \cite[Examples 3.8(2)]{Les85} satisfies $\fm^3=0$. 
\item[(b)] Tachikawa's Conjecture \ref{Tachikawa} holds true for $R$ (see Asashiba \cite{Asa91}, Avramov-Buchweitz-\c{S}ega \cite{ABS05}, Huneke-\c{S}ega-Vraciu \cite[Theorem 2.11]{HSV04}). 
\item[(c)] If $R$ is not Gorenstein, then the Bass sequence $\{\mu_i(R)\}_{i\geq 0}$ is increasing and has termwise exponential growth (see Christensen-Striuli-Veliche \cite[Theorem 5.1, Theorem A.4]{CSV10}).
\end{itemize}
\end{remark}

To conclude this section, we propose two problems for further investigation.

\begin{problem} Let $(R, \fm, k)$ be a Noetherian local ring.
\begin{itemize}
\item[(a)] Suppose $\fm^3=0$. It is interesting to explore non-trivial direct sum decompositions or to identify non-trivial direct summands of the syzygies of the residue field of $R$, especially in the case $R$ does not satisfy the syzygy direct summand assumption.
\item[(b)] Suppose $R$ satisfies the syzygy direct summand assumption $(\ast)$. Is it true that syzygies of any finitely generated $R$-module satisfy similar property?
\end{itemize}
\end{problem}


\section{Syzygy direct sum decomposition and characterization of Golod rings}

Unless otherwise specified, $(R, \fm, k)$ denotes again a Noetherian local ring and $e=\dim_k(\fm/\fm^2)$ its embedding dimension. We denote by $c=e-\depth(R)$ the codepth of $R$.

Theorem \ref{generalGolod} provides a direct sum decomposition of syzygies of the residue field of a Golod ring in terms of lower syzygies. Such decompositions are rare among local rings and one may ask whether they are special to Golod rings. The answer is affirmative, and we provide a proof in this section.

\begin{proposition} \label{p31}
Let $x\in \fm\setminus \fm^2$ be a non-zero divisor of $R$. Then $R$ is Golod if and only if so is $R/xR$.
\end{proposition}
\begin{proof}
We refer to \cite[Proposition 5.2.4]{Avr98}.
\end{proof}

We recall the ``depth sensitivity'' of Koszul homology.

\begin{lemma} \label{l31}
Let $M$ be a finitely generated $R$-module. The equality 
\begin{equation*}
e-\depth_R M=\max\{i \mid \rh_i(M)\not=0\}
\end{equation*}
holds. In particular, 
\begin{equation*}
c=\max\{i\mid \rh_i(R)\not=0\}.
\end{equation*}

\end{lemma}

\begin{proof}
See \cite[Theorem 1.6.17]{BH98}.
\end{proof}

For an $R$-module $M$, $\soc_R(M)$ denotes the socle of $M$, that is, $\soc_R(M)=\{x\in M\mid \mathfrak{m}x=0\}$.
We have a lemma.

\begin{lemma} \label{l32}
Let $M$ be a finitely generated $R$-module. Let $n\ge 1$ be an integer. The following assertions hold:
\begin{enumerate}[\rm(1)]
\item $\rh_i(\syz_n^RM) \le \beta_{n-1}(M)\rh_i(R)+\rh_{i+1}(\syz_{n-1}^RM)$ for any $i>0$.
\item $\rh_0(\syz_n^RM) \le \rh_1(\syz_{n-1}^RM)$.
\item $\rh_0(\syz_n^RM)=\beta_n(M)$.
\item If $n\ge 2$,  $\rh_e(\syz_n^RM)=\beta_{n-1}(M)\rh_e(R)$.
\end{enumerate}
\end{lemma}

\begin{proof}
(1) The $R$-modules $\syz_n^R M$ and $\syz_{n-1}^RM$ fit into a short exact sequence 
\[
0 \to \syz_n^R M \to R^{\oplus \beta_{n-1}(M)} \xrightarrow[]{\alpha} \syz_{n-1}^RM \to 0.
\]
Tensoring the sequence with the Koszul complex $K_\bullet$ and taking the homologies, we obtain a long exact sequence
\begin{align*}
0 & \to \mathrm{H}_e(\underline{x}; \syz_n^R M) \xrightarrow[]{\gamma_e} \mathrm{H}_e(\underline{x},R)^{\oplus \beta_{n-1}(M)} \to \mathrm{H}_e(\underline{x}; \syz_{n-1}^RM ) \\
& \to \mathrm{H}_{e-1}(\underline{x}; \syz_n^R M) \to \mathrm{H}_{e-1}(\underline{x},R)^{\oplus \beta_{n-1}(M)} \to \mathrm{H}_{e-1}(\underline{x}; \syz_{n-1}^RM )\\
&\cdots \\
& \to \mathrm{H}_1(\underline{x}; \syz_n^R M) \to \mathrm{H}_1(\underline{x},R)^{\oplus \beta_{n-1}(M)} \to \mathrm{H}_1(\underline{x}; \syz_{n-1}^RM )\\
& \to \mathrm{H}_0(\underline{x}; \syz_n^R M) \to \mathrm{H}_0(\underline{x},R)^{\oplus \beta_{n-1}(M)} \xrightarrow[]{\alpha_0} \mathrm{H}_0(\underline{x}; \syz_{n-1}^RM ) \to 0.
\end{align*}
Therefore, one has an inequality
\[
\dim_k \mathrm{H}_i(\underline{x};\syz_n^R M) \le \dim_k \mathrm{H}_i(\underline{x};R)^{\oplus \beta_{n-1}(M)} + \dim_k \mathrm{H}_{i+1}(\underline{x}; \syz_{n-1}^RM).
\]
This shows (1).

(2) $\alpha \otimes_R k$ is an isomorphism, and so is $\alpha_0$.
Therefore, one has a surjection $\mathrm{H}_1(\underline{x};\syz_{n-1}^RM) \to \mathrm{H}_0(\underline{x};\syz_n^R M)$. Taking the dimension of the $k$-vector spaces on the both sides, it yields the inequality.

(3) Observe that $\mathrm{H}_0(\underline{x};\syz_n^R M)$ is isomorphic to $\syz_n^R M \otimes_R k$. Thus $\rh_0(\syz_n^R M)$ is equal to $\beta_{n}(M)$.

(4) Since the matrix of the map $R^{\oplus \beta_{n-1}(M)}\rightarrow R^{\oplus \beta_{n-2}(M)}$ has all entries in $\fm$, the socle $\soc(R^{\oplus \beta_{n-1}(M)})$ is included in the kernel, i.e., $\soc(R^{\oplus \beta_{n-1}(M)})\subseteq \syz^R_n(M)$.
The opposite inclusion $\soc_R(\syz_n^RM) \subseteq \soc_R(R^{\oplus \beta_{n-1}(M)})$ is obvious.
We get the equality $\soc_R(\syz_n^RM)=\soc_R(R^{\oplus \beta_{n-1}(M)})$.
As a basic property of Koszul homology, we have a natural isomorphisms
$$\soc(M) \simeq 0:_M\fm \simeq H_e(\underline x; M).$$
It follows that the map
$
\gamma_e \colon \rH_e(\underline{x}; \syz_n^R M) \to \rH_e(\underline{x};R)^{\oplus \beta_{n-1}(M)}
$
 is an isomorphism.
\end{proof}
%
%
%

For simplicity, let $\tilde{\rh}_0(R)=0$ and $\tilde{\rh}_i(R)=\rh_i(R)$ for $i>0$.
The dimensions of the Koszul homology of the 0th, 1st, and 2nd syzygies of the residue field can be explicitly computed.

\begin{lemma} \label{l35}
Let $R$ be a Noetherian local ring.
Let $i\ge 0$ be an integer.
The following assertions hold:
\begin{enumerate}[\rm(1)]
\item $\rh_i(k)=\binom{e}{i}$.
\item $\rh_i(\syz_1^Rk)=\binom{e}{i+1}+\tilde{\rh}_i(R)$.
\item 
$\rh_i(\syz_2^Rk)=
\binom{e}{i+2}+\tilde{\rh}_{i+1}(R)+e\tilde{\rh}_i(R)$.
\end{enumerate}
\end{lemma}

\begin{proof}
(1) The $R$-complex $K(\underline{x})\otimes_R k$ has zero differentials.
Thus, $\mathrm{H}_i(\underline{x}; k)=K_i(\underline{x})\otimes_R k$.
In particular, $\rh_i(k)$ is equal to $\binom{e}{i}$.

(2) Note that the $R$-module $\syz_1^R k$ is just the maximal ideal $\fm$.
Due to Lemma \ref{l32}, we have $$\rh_0(\syz_1^R k)=\beta_1(k)=e+\tilde{\rh}_0(R).$$
Let $i$ be positive.
For an $R$-module $M$, we denote by $\mathrm{Z}_i(\underline{x}; M)$ the kernel of the $i$-th differential map $K_i(\underline{x})\otimes_R M \to K_{i-1}(\underline{x})\otimes_R M$ of the Koszul complex tensored with $M$.
Since $\underline{x}$ is a system of minimal generators of $\fm$, the submodule $\mathrm{Z}_i(\underline{x};R)$ of $K_i(\underline{x})$ is contained in $\fm K_i(\underline{x})$ \cite[Lemma 4.1.6]{Avr98}.
The canonical inclusion $\fm \hookrightarrow R$ induces a term-wise inclusion $K_*(\underline{x})\otimes_R \fm \hookrightarrow K_*(\underline{x})$ of complexes.
Via this inclusion, the complex  $K_*(\underline{x})\otimes_R \fm$ is identified with the subcomplex $\fm K_*(\underline{x})$ of the Koszul complex $K_*(\underline{x})$.
Under the identification, $\mathrm{Z}_i(\underline{x};M)$ coincides with $\mathrm{Z}_i(\underline{x};R)\cap  (K_i(\underline{x})\otimes_R \fm)=\mathrm{Z}_i(\underline{x};R)$.
Therefore, the induced homomorphism $\mathrm{H}_i(\underline{x}; \fm) \to \mathrm{H}_i(\underline{x};R)$ is surjective.
It follows that there is a short exact sequence
\[
0 \to \mathrm{H}_{i+1}(\underline{x};k) \to \mathrm{H}_i(\underline{x};\fm) \to \mathrm{H}_i(\underline{x};R) \to 0.
\]
The dimension $\rh_i(\fm)$ of the middle term is thus equal to $\rh_{i+1}(k)+\rh_i(R)$, which is $\binom{e}{i+1}+\tilde{\rh}_i(R)$ by (1).

(3) Due to Lemma \ref{l32} and \cite[Theorem 2.3.2]{BH98}, we have $$\rh_0(\syz_2^R k)=\beta_2(k)=\binom{e}{2}+\rh_1(R)+e\tilde{\rh}_0(R).$$
Assume that $i$ is positive.
The homomorphism $\syz_2^Rk \hookrightarrow R^{\oplus \beta_1(k)}$ factors through the inclusion $R^{\oplus \beta_1(k)}\otimes_R \fm \hookrightarrow R^{\oplus \beta_1(k)}$.
So we get a commutative diagram
\[
\xymatrix@C+1pc{
\syz_2^R k \ar[r] \ar[rd] & R^{\oplus \beta_1(k)} \ar[r]^{d_1} \ar[rd] & R\\
& \fm^{\oplus \beta_1(k)} \ar[r]^{d_1\otimes \fm} \ar[u] & \fm \ar[u]
}
\]
The first differential map $d_1$ in the minimal free resolution of $k$ is represented by the $1\times e$ matrix $[x_1~\dots x_e]$.
Since $\mathrm{H}_i(\underline{x};\fm)$ is annihilated by $\fm$, the induced homomorphism 
$$[x_1,\dots,x_e]=\mathrm{H}_i(d_1\otimes \fm)\colon \mathrm{H}_i(\underline{x}; \fm^{\oplus \beta_1(k)}) \to \mathrm{H}_i(\underline{x}; \fm)$$
 is a zero homomorphism.
On the other hand, as we argued in (2), we know that the induced homomorphism $\mathrm{H}_i(\underline{x};\fm^{\oplus \beta_1(k)}) \to \mathrm{H}_i(\underline{x}; R^{\oplus \beta_1(k)})$ is surjective.
Consequently, the homomorphism $\mathrm{H}_i(\underline{x};R^{\oplus \beta_1(k)}) \to \mathrm{H}_i(\underline{x};\fm)$ is a zero homomorphism.
From this, one may derive that there is a short exact sequence
\[
0 \to \mathrm{H}_{i+1}(\underline{x}; \fm) \to \mathrm{H}_i(\underline{x}; \syz_2^R k) \to \mathrm{H}_i(\underline{x}; R^{\oplus \beta_1(k)}) \to 0,
\]
which yields the equality $\rh_i(\syz_2^Rk)=\binom{e}{i+2}+\tilde{\rh}_{i+1}(R)+\beta_1(k)\tilde{\rh}_i(R)$.
\end{proof}

\begin{lemma} \label{l36}
Let $(R,\fm,k)$ be a Noetherian local ring and $x\in \fm \setminus \fm^2$ be a non-zero divisor on $R$.
Let $M$ be a finitely generated $R$-module such that $x$ is regular on $M$.
Then for any $n\ge 0$, we have $\rh_n^R(M)=\rh_n^{R/xR}(M/xM)$.
In particular, we have $\tilde{\rh}_n^R(R)=\tilde{\rh}_n^{R/xR}(R/xR)$ for any $n\ge 0$.
\end{lemma}

\begin{proof}
We refer to \cite[Proposition 1.6.7 and Corollary 1.6.13]{BH98}.
\end{proof}

\begin{proposition} \label{l37}
Let $(R,\fm,k)$ be a Noetherian local ring of codepth $c$. For integers $n\ge 0$ and $m\ge c$, we have
\begin{equation*}
\rh_m(\syz_n^Rk)=\beta_{n-1}^R(k)\tilde{\rh}_m(R)+\binom{e}{n+m}.
\end{equation*}
Here we treat $\beta_{-1}^R(k)=0$.
\end{proposition}

\begin{proof}
If $n=0$ or $n=1$, the assertion follows from Lemma \ref{l35}. Hence, we may assume $n\geq 2$.

The proposition is proved by induction on $t=\depth(R)$. If $t=0$ then the assertion follows from Lemma \ref{l32}(4). Suppose $t>0$. Pick an element $x\in \fm\setminus \fm^2$ such that $x$ is regular on $R$. By Lemma \ref{l36} and \ref{l27}, we have
\begin{equation*}
\rh_m^R(\syz_n^R k)=\rh_m^{R/xR}(\syz_n^{R/xR}k)+\rh_m^{R/xR}(\syz_{n-1}^{R/xR}k).
\end{equation*}
Since $\depth R/xR=t-1$, the induction hypothesis applies.
Therefore, we get 
\begin{align*}
\rh_m^{R/xR}(\syz_n^{R/xR}k)&=\beta_{n-1}^{R/xR}(k)\tilde{\rh}_m^{R/xR}(R/xR)+\binom{e-1}{n+m},\text{ and} \\ \rh_m^{R/xR}(\syz_{n-1}^{R/xR}k)&=\beta_{n-2}^{R/xR}(k)\tilde{\rh}_m^{R/xR}(R/xR)+\binom{e-1}{n+m-1}.
\end{align*}
Lemma \ref{l27} and \ref{l36} say $\beta_{n-1}^R(k)=\beta_{n-1}^{R/xR}(k)+\beta_{n-2}^{R/xR}(k)$ and $\tilde{\rh}_m^R(R)=\tilde{\rh}_m^{R/xR}(R/xR)$ respectively.
We now obtain the desired equality:
\begin{align*}
\rh_m(\syz_n^Rk)&=\beta_{n-1}^{R/xR}(k)\tilde{\rh}_m^{R/xR}(R/xR)+\beta_{n-2}^{R/xR}(k)\tilde{\rh}_m^{R/xR}(R/xR)+\binom{e-1}{n+m}+\binom{e-1}{n+m-1}\\
&=\beta_{n-1}^R(k)\tilde{\rh}_m(R)+\binom{e}{n+m}.
\end{align*}
\end{proof}

Combining the lemmas yields the following useful consequences.

\begin{corollary}\label{cor32}
Let $(R,\fm,k)$ be a Noetherian local ring with $c>0$.
\begin{enumerate}[\rm(1)]
\item 
For any $n\ge 0$, we have
\begin{equation} \label{eq37}
\beta_n^R(k) \le \sum_{j=1}^c\beta_{n-j-1}^R(k)\rh_j(R)+\binom{e}{n}.
\end{equation}
\item Let $n\ge 0$ be fixed. 
The following conditions are equivalent:
\begin{enumerate}[\rm(a)]
\item The equality in \eqref{eq37} holds.
\item The equality
\begin{equation*}
\rh_i(\syz_{n-i}^Rk)=\left\{\begin{array}{ll}
\rh_1(\syz_{n-1}^Rk) & (i=0)\\
\beta_{n-i-1}^R(k)\rh_i(R)+\rh_{i+1}(\syz_{n-i-1}^Rk) & (i>0)
\end{array}
\right.
\end{equation*}
holds for any $i=0,\dots,n-1$.
\item The equality
\begin{equation}
\rh_i(\syz_{n-i}^Rk)=
\sum_{j=i}^{c}\beta_{n-j-1}^R(k)\tilde{\rh}_{j}(R)+\binom{e}{n}
\end{equation}
holds for any $i=0,\dots,n-1$.
\end{enumerate}
Here we treat $\syz_{-a}^R(k)=0$ and $\beta_{-a}^R(k)=0$ for any $a>0$.
\end{enumerate}
\end{corollary}
\begin{proof}
First, we deal with the case $n\ge c$.
In view of Lemma \ref{l32} the following inequalities hold true 
\begin{align*}
\beta_n^R(k) &\le \rh_1(\syz_{n-1}^Rk)\\
&\le \beta_{n-2}^R(k)\rh_1(R)+\rh_2(\syz_{n-2}^Rk)\\
&\le \ldots\\
&\le \sum_{j=1}^{c-1}\beta_{n-j-1}^R(k)\rh_j(R)+\rh_{c}(\syz_{n-c}^Rk)\\
&= \sum_{j=1}^c\beta_{n-j-1}^R(k)\rh_j(R)+\binom{e}{n}.
\end{align*}
Here the last equality is obtained by applying Lemma \ref{l37}.

In the case $n=0$, the assertion is obvious.

Next, we consider the case $0<n<c$.
A similar argument as above yields the following inequalities:
\begin{align*}
\beta_n^R(k) &\le \rh_1(\syz_{n-1}^Rk)\\
&\le \beta_{n-2}^R(k)\rh_1(R)+\rh_2(\syz_{n-2}^Rk)\\
&\le \ldots\\
&\le \sum_{j=1}^{n-1}\beta_{n-j-1}^R(k)\rh_j(R)+\rh_{n}(\syz_0^Rk)\\
&= \sum_{j=1}^{n-1}\beta_{n-j-1}^R(k)\rh_j(R)+\binom{e}{n}.
\end{align*}
Here the last equality is obtained by applying Lemma \ref{l35} (1).
Adding $0=\sum_{j=n}^c \beta_{n-j-1}^R(k)\rh_j(R)$ to both sides, we obtain the desired inequality.

The equivalence (a)$\Leftrightarrow$(b) in (2) then follows directly. The equivalence (b)$\Leftrightarrow$(c) in (2) follows from the basic recurrence formula.
\end{proof}

Recall that $t=\depth R$ and $c=e-t$.
Therefore, if $t=0$, we have $e=c$.

\begin{proposition} \label{p39a}
Let $(R,\fm,k)$ be a Noetherian local ring with $t=0$ and $e>0$.
For integers $l,m,n$, consider the following conditions:
\begin{equation} \label{eqB}
\beta_n^R(k)=\sum_{j=1}^e\beta_{n-j-1}^R(k)\rh_j(R)+\binom{e}{n}, \tag{$\text{B}_n^R$}
\end{equation}
and
\begin{equation} \label{eqH}
\rh_m(\syz_{e+l+1}^Rk)=\sum_{j=1}^e\rh_m(\syz_{e+l-j}^Rk)\rh_j(R). \tag{$\text{H}_{m,l}^R$}
\end{equation}
Let $a,b \ge 0$ be integers such that $a\le e$.
\begin{enumerate}[\rm(1)]
\item Assume that $\emph{(B}_n^R\emph{)}$ holds for any $n\ge a+b$.
Then $\emph{(H}_{a,b}^R\emph{)}$ holds.
\item Assume that $\emph{(B}_n^R\emph{)}$ holds for any $n\ge a+1$, and that $\emph{(H}_{a,0}^R\emph{)}$ holds.
Then $\emph{(B}_a^R\emph{)}$ holds.
\end{enumerate}
\end{proposition}

\begin{proof}
(1): 
In view of Corollary \ref{cor32}, it follow from $\text{(B}_{n}^R\text{)}$ for $n\ge a+b$ that
\begin{equation}
\rh_a(\syz_{e+b+1}^Rk)=\sum_{j=a}^e\beta_{a+e+b-j}^R(k)\tilde{\rh}_j(R),
\end{equation}
and that
\begin{equation} \label{eq36}
\rh_a(\syz_{e+b-s}^Rk)=\sum_{j=a}^e\beta_{a+e+b-s-j-1}^R(k)\tilde{\rh}_{j}(R)+\binom{e}{a+e+b-s} \quad \text{for any $s\le e$.}
\end{equation}
Therefore, we obtain
\begin{align} \label{eqe37}
&\rh_a(\syz_{e+b+1}k)-\sum_{s=1}^e\rh_a(\syz_{e+b-s}k)\rh_s(R)\\
&=\sum_{j=a}^e\left(\beta_{a+e+b-j}(k)\tilde{\rh}_j(R)-\sum_{s=1}^e\beta_{a+e+b-s-j-1}(k)\tilde{\rh}_j(R)\rh_s(R)\right)-\sum_{s=1}^e\binom{e}{a+e+b-s}\rh_s(R) \notag\\
&=\sum_{j=a}^{e}\left(\beta_{a+e+b-j}(k)-\sum_{s=1}^e\beta_{a+e+b-s-j-1}(k)\rh_s(R)-\binom{e}{a+e+b-j}\right)\tilde{\rh}_j(R) \notag\\
&+\sum_{j=a}^e\binom{e}{a+b+e-j}\tilde{\rh}_j(R)-\sum_{s=1}^e\binom{e}{a+e+b-s}\rh_s(R) \notag
\end{align}
Substituting $\text{(B}_{n}^R\text{)}$ for $n=a+b,\dots,a+b+e-1$ into \eqref{eqe37} yields
\begin{align}
&\rh_a(\syz_{c+b+1}k)-\sum_{s=1}^c\rh_a(\syz_{c+b-s}k)\rh_s(R)=-\sum_{j=1}^{a-1}\binom{e}{a+b+e-j}\rh_j(R)=0.
\end{align}
This means that $\text{(H}_{a,b}\text{)}$ holds.

(2): In view of Corollary \ref{cor32}, it follows from $\text{(B}_{n}^R\text{)}$ for $n\ge a+1$ that
\begin{equation}
\rh_a(\syz_{e+1}k)=\sum_{j=a}^e\beta_{a+e-j}(k)\tilde{\rh}_j(R),
\end{equation}
and that
\begin{equation} \label{eq36a}
\rh_a(\syz_{e-s}k)=\sum_{j=a}^e\beta_{a+e-s-j-1}(k)\tilde{\rh}_{j}(R)+\binom{e}{a+e-s} \quad \text{for any $s< e$.}
\end{equation}
Note that, due to Lemma \ref{l35} (1), the equality \eqref{eq36a} also holds for $s=e$.
Therefore, as a similar argument to (1), we obtain
\begin{align} \label{eqe37a}
&\rh_a(\syz_{e+1}k)-\sum_{s=1}^e\rh_a(\syz_{e-s}k)\rh_s(R)\\
&=\sum_{j=a}^e\left(\beta_{a+e-j}(k)\tilde{\rh}_j(R)-\sum_{s=1}^e\beta_{a+e-s-j-1}(k)\tilde{\rh}_j(R)\rh_s(R)\right)-\sum_{s=1}^e\binom{e}{a+e-s}\rh_s(R) \notag\\
&=\left(\beta_{a}(k)-\sum_{s=1}^e\beta_{a-s-1}(k)\rh_s(R) -\binom{e}{a}\right)\rh_e(R) \notag\\
&+\sum_{j=a}^{e-1}\left(\beta_{a+e-j}(k)-\sum_{s=1}^e\beta_{a+e-s-j-1}(k)\rh_s(R)-\binom{e}{a+e-j}\right)\tilde{\rh}_j(R) \notag\\
&+\sum_{j=a}^e\binom{e}{a+e-j}\tilde{\rh}_j(R)-\sum_{s=1}^e\binom{e}{a+e-s}\rh_s(R) \notag
\end{align}
Substituting $\text{(B}_{n}^R\text{)}$ for $n\ge a+1$ and $\text{(H}_{a,0}\text{)}$ into \eqref{eqe37a} yields
\begin{align}
0=\left(\beta_{a}(k)\rh_e(R)-\sum_{s=1}^e\beta_{a-s-1}(k)\rh_s(R) -\binom{e}{a}\right)\rh_e(R).
\end{align}
Since $\rh_e(R)\not=0$, we achieve $\text{(B}_{a}^R\text{)}$.
\end{proof}

We state the following numerical criterion of Golodness directly derived from the definition.

\begin{lemma} \label{l26}
A Noetherian local ring $R$ is Golod if and only if the equality \eqref{eqB}
holds true for any integer $n\ge 0$.
\end{lemma}

\begin{proof}
In view of \eqref{equ21}, $R$ is Golod exactly when the equality below holds:
\[
\mathrm{P}_k^R(t)=(1+t)^e+\mathrm{P}_k^R(t)\sum_{j=1}^c\rh_j(R)t^{j+1}.
\]
The equality $\text{(B}_n^R\text{)}$ is equivalent to this.
\end{proof}

Now comes the main result of this section. 

\begin{theorem}\label{characterizationGolod}
Let $(R,\fm,k)$ be a Noetherian local ring of embedding dimension $e$, and let $K_{\bullet}$ be the Koszul complex on a minimal set of generators of $\fm$.
Put $\rh_i(R)=\dim_k \mathrm{H}_i(K_\bullet)$.
The following statements are equivalent.
\begin{itemize}
\item[(a)] The ring $R$ is Golod.
\item[(b)] There is an isomorphism
\begin{equation} \label{generalGolod23}
\syz_{e+1}^R(k) \simeq \bigoplus_{j = 1}^{c} \syz_{e-j}^R(k)^{\oplus \rh_{j}(R)}.
\end{equation}
\item[(c)] For each $m\ge 0$, there are isomorphisms
\begin{equation} \label{eq315}
\syz_{e+1+m}^R(k) \simeq \bigoplus_{j=1}^{c} \syz_{e-j+m}^R(k)^{\oplus  \rh_{j}(R)}.
\end{equation}
\end{itemize}
\end{theorem}

\begin{proof}
The implication (a) $\Rightarrow$ (b) was proved by Cuong-Dao-Eisenbud-Kobayashi-Polini-Ulrich \cite[Theorem 1.1]{CDEKPU25}.
The implication (b) $\Leftrightarrow$ (c) is clear. 
We will prove (c) $\Rightarrow$ (a).
We may assume $c>0$.

Let $\underline{x}=x_1,\dots,x_t$ be a maximal $R$-regular sequence inside $\fm\setminus \fm^2$.
We set $S=R/(x_1,\dots,x_t)R$.
Remark that $S$ has depth $0$ and embedding dimension $c=e-t$, and that $\rh_j^R(R)=\rh_j^{S}(S)$ following Lemma \ref{l36}. Thanks to Proposition \ref{p31}, it is enough to show that $S$ is Golod.
Applying Corollary \ref{c26} to both sides of \eqref{eq315}, for each $m\ge 0$, we get an isomorphism of $S$-modules:
\begin{equation} \label{eq316}
\bigoplus_{l=0}^{t} \syz^{S}_{c+m+l+1}k^{\oplus \binom{t}{l}} \cong \bigoplus_{l=0}^t\left(\bigoplus_{j=1}^c\syz_{c+m+l-j}^{S}k^{\oplus \rh_j^S(S)}\right)^{\oplus \binom{t}{l}}.
\end{equation}
By taking Betti numbers of both sides of \eqref{eq316}, it follows that the equality
\begin{equation} \label{eq317}
\sum_{l=0}^t\binom{t}{l}\beta_{n+l}^{S}(k)=\sum_{l=0}^t\binom{t}{l}\sum_{j=1}^c\beta_{n-j-1+l}^{S}(k)\rh_j^S(S)
\end{equation}
holds for any $n\ge c+1$.
In view of Corollary \ref{cor32}, it follows from \eqref{eq317} that the equality
$\text{(B}_n^S\text{)}$
holds for any $n\ge c+1$.
In the rest, we prove by descending induction on $n$ that the equality $\text{(B}_n^S\text{)}$
holds for any $0 \le n \le c+1$.
The case of $n=c+1$ was done.
Suppose that $n \le c$. Also, as the induction hypothesis, we suppose that $\text{(B}_m^S\text{)}$ holds for any $m\in\{n+1,\dots,c+1\}$.
By taking Koszul homologies of both sides of \eqref{eq316}, it follows that the equality
\begin{equation} \label{eq319}
\sum_{l=0}^t\binom{t}{l}\rh_n^S(\syz_{c+l+1}^Sk)=\sum_{l=0}^t\binom{t}{l}\sum_{j=1}^c\rh_{n}^{S}(\syz_{c+l-j}^Sk)\rh_j^S(S)
\end{equation}
holds.
According to Proposition \ref{p39a} (1), we see that the equality 
\begin{equation} \label{eqHa}
\rh_n(\syz_{c+l+1}^Sk)=\sum_{j=1}^c\rh_n(\syz_{c+l-j}^Sk)\rh_j^S(S) \tag{$\text{H}_{n,l}^S$}
\end{equation}
holds for any $l \ge 1$. 
Substituting this into \eqref{eq319} yields that the equality $\text{(H}_{n,0}^S\text{)}$ holds.
Then, due to Proposition \ref{p39a} (2), we conclude that $\text{(B}_n^S\text{)}$ holds, and hence that the induction proceeds.
Finally, Lemma \ref{l26} implies that $S$ is Golod. 
\end{proof}

The decomposition of syzygies in Theorem \ref{characterizationGolod} leads to several questions concerning the characterization of Golod rings.

\begin{question}
Let $(R, \fm, k)$ be a Noetherian local ring of embedding dimension $e$. Assume that for some $m\ge 1$, there is an isomorphism
\begin{equation}
\syz_{e+1+m}^R(k) \simeq \bigoplus_{j=1}^{c} \syz_{e-j+m}^R(k)^{\oplus  \rh_{j}(R)}.
\end{equation}
Is $R$ a Golod ring?
\end{question}

For a collection $\mathcal{S}$ of finitely generated $R$-modules, $\operatorname{add}(\mathcal{S})$ denotes the collection of all $R$-modules $X$ such that there exist finitely many modules $M_1,\dots,M_n\in \mathcal{S}$ and integers $a_1,\dots,a_n \ge 1$ such that $X$ is isomorphic to a direct summand of $M_1^{\oplus a_1}\oplus \cdots \oplus M_n^{\oplus a_n}$.

Let $R$ be a Golod local ring. For $n\ge 0$, set $\mathcal{S}_n=\{\syz_0^Rk,\dots,\syz_{n}^Rk\}$. According to \cite[Theorem 1.1]{CDEKPU25}, we know that $\syz_{e+1}^R(k)$ belongs to $\operatorname{add}(\mathcal{S}_{e-1})$. One may ask about the converse.

\begin{question}
Let $(R, \fm, k)$ be a Noetherian local ring of embedding dimension $e$. Assume that for some $n\ge e$, $\syz_{n+1}^R(k)$ belongs to $\operatorname{add}(\mathcal{S}_{n-1})$. Is $R$ a Golod ring?
\end{question}


\end{document}